\documentclass[11pt,reqno]{amsart}

\usepackage{amsmath,amssymb,amsfonts,amsthm}
\usepackage{accents}
\usepackage{enumerate,enumitem}
\usepackage[top=30truemm,bottom=30truemm,left=30truemm,right=30truemm]{geometry}
\usepackage[breaklinks=true,colorlinks=true,linkcolor=blue,citecolor=blue,filecolor=blue,urlcolor=blue]{hyperref}
\hypersetup{
pdftitle={Order of Zeros of Dedekind Zeta Functions},
pdfauthor={Daniel Hu, Ikuya Kaneko, Spencer Martin, and Carl Schildkraut},
linktocpage}
\usepackage{mathrsfs}
\usepackage{mathtools}
\usepackage{etoolbox}
\usepackage{todonotes}
\usepackage[capitalise]{cleveref}

\newtheorem{theorem}{Theorem}[section]
\newtheorem{lemma}[theorem]{Lemma}
\newtheorem{corollary}[theorem]{Corollary}

\newtheorem{conjecture}[theorem]{Conjecture}
\theoremstyle{definition}
\newtheorem{definition}[theorem]{Definition}
\newtheorem{remark}[theorem]{Remark}

\newcommand{\C}{\mathbb{C}}
\newcommand{\F}{\mathbb{F}}

\newcommand{\Q}{\mathbb{Q}}

\newcommand{\Z}{\mathbb{Z}}
\newcommand{\oO}{\mathcal O}
\newcommand{\idp}{\mathfrak p}
\DeclareMathOperator{\Ind}{Ind}
\DeclareMathOperator{\Gal}{Gal}
\DeclareMathOperator*{\ord}{ord}
\DeclareMathOperator{\frob}{Frob}
\DeclareMathOperator{\AGL}{AGL}
\DeclareMathOperator{\PSL}{PSL}
\DeclareMathOperator{\SL}{SL}
\DeclareMathOperator{\SUT}{SUT}
\DeclareMathOperator{\PSUT}{PSUT}
\DeclareMathOperator{\Sz}{Sz}

\begin{document}

\title{Order of Zeros of Dedekind Zeta Functions}

\author{Daniel Hu}
\address{Department of Mathematics, Princeton University, Fine Hall, Washington Road, Princeton NJ 08544-1000, USA}
\email{danielhu@princeton.edu}

\author{Ikuya Kaneko}
\address{Department of Mathematics, California Institute of Technology, 1200 E California Blvd, Pasadena, CA 91125, USA}
\email{ikuyak@icloud.com}
\urladdr{\href{https://sites.google.com/view/ikuyakaneko/}{https://sites.google.com/view/ikuyakaneko/}}

\author{Spencer Martin}
\address{Department of Mathematics, University of Virginia, 141 Cabell Drive, Kerchof Hall, Charlottesville, VA 22904, USA}
\email{sm5ve@virginia.edu}

\author{Carl Schildkraut}
\address{Department of Mathematics, Massachusetts Institute of Technology, 77 Massachusetts Avenue, Cambridge, MA 02139-4307, USA}
\email{carlsc@mit.edu}

\subjclass[2020]{Primary: 11R42; Secondary: 20C15}

\keywords{}

\date{\today}

\begin{abstract}
Answering a question of Browkin, we provide a new unconditional proof that~the Dedekind zeta function of a number field $L$ has infinitely many nontrivial zeros of multiplicity at least 2 if $L$ has a subfield $K$ for which $L/K$ is a nonabelian Galois extension. We~also~extend this to zeros of order 3 when $\Gal(L/K)$ has an irreducible representation of degree at least~3, as predicted by the Artin holomorphy conjecture.
\end{abstract}

\maketitle

\section{Introduction}\label{introduction}
The study of the order of zeros of $L$-functions is one of the central problems in number~theory. It is conjectured that all nontrivial zeros of the Riemann zeta function lie on the critical line $\Re(s) = \frac{1}{2}$ and are simple. On the other hand, there exist number fields $L$ such that the Dedekind zeta function $\zeta_L(s)$ has nontrivial zeros of higher order. This is due to the decomposition of Dedekind zeta functions as a product of Artin $L$-functions. The Artin holomorphy conjecture predicts that Artin $L$-functions associated to nontrivial irreducible representations are entire. Assuming this conjecture, if $L/K$ is a nonabelian Galois extension of number~fields, then $\zeta_L(s)$ has infinitely many nontrivial zeros of higher order.

If $L/\Q$ is Galois and we further assume that no two Artin $L$-functions associated to irreducible characters of $\Gal(L/\Q)$ share nontrivial zeros (with the possible exception of $s = \frac{1}{2}$), as well as that all such zeros are simple, then one can be more precise: the highest order~nontrivial zeros of $\zeta_L(s)$ away from $s = \frac{1}{2}$ have order equal to the greatest degree of any~irreducible representation of $\Gal(L/\Q)$ and there are infinitely many such zeros.

Since the Artin holomorphy conjecture is known for specific Galois groups, it has long~been known that there are infinitely many~cases of zeros of higher order away from $s = \frac{1}{2}$. Browkin \cite{browkin-2013} studied one such family of Galois groups. His example concerns the affine group over~a finite field, namely the matrix group
\begin{equation}\label{eq:aff-gp-def}
    \AGL_1(\mathbb{F}_q) \coloneqq \bigg\{
    \begin{pmatrix}
        a & b \\ 0 & 1
    \end{pmatrix} \bigg|\,\, a \in \mathbb{F}_q^\times, b \in \mathbb{F}_q \bigg\}.
\end{equation}
Each such group possesses only one irreducible character of degree greater than one (which has degree $q-1$) induced by one-dimensional characters on a subgroup of~$\AGL_1(\F_q)$. As~a result (see~\cref{cor:hol-artin-induct}), the Artin holomorphy conjecture is known for the Artin $L$-function corresponding to this character. Hence, for every Galois extension $L/K$ of number fields~with Galois group~$\AGL_1(\F_q)$, the Dedekind zeta function $\zeta_{L}(s)$ has infinitely many zeros of~multiplicity at least $q - 1$ in the critical strip~$0 < \Re(s) < 1$. This led Browkin to ask if $\zeta_L(s)$ always has higher order nontrivial zeros whenever $L/K$ is nonabelian \cite[Section 7]{browkin-2013}.

An alternative, more general approach is to study the holomorphy of a family of Artin $L$-functions at a given point. This is the approach taken by Stark \cite[Theorem 3]{stark1974some} who~showed that if $L/K$ is Galois and $\rho$ is a simple zero of $\zeta_L(s)$, then $L(s, \chi, L/K)$ is holomorphic at~$s = \rho$ for every character $\chi$. Stark's result has been refined by Foote and Kumar Murty \cite{foote_murty_1989} and by~Foote and Wales \cite{footeorder2} to treat higher order zeros when $L/K$ is a solvable extension. As a corollary of the work of Stark, we obtain the following theorem.

\begin{theorem}\label{thm:mainthm}
If $L/K$ is a nonabelian Galois extension of number fields, the Dedekind zeta function $\zeta_{L}(s)$ has~infinitely many nontrivial zeros with multiplicity greater than $1$.
\end{theorem}

The goal of this paper is to prove this result by wholly different means: we non-constructively consider the class of Galois groups for~which the conclusion of \cref{thm:mainthm} holds~and establish that this encompasses all nonabelian finite groups.

Although the example of Browkin allows for explicit lower bounds for the multiplicities~of the nontrivial zeros of interest, neither proof of \cref{thm:mainthm} is able to prescribe such multiplicities predicted by the Artin holomorphy conjecture. However, we can match this prediction~for zeros of order 3.

\begin{theorem}\label{thm:degree3}
Let $L/K$ be a Galois extension of number fields. If $\Gal(L/K)$ has an irreducible representation of degree at least 3, then the Dedekind zeta function $\zeta_{L}(s)$ has~infinitely many nontrivial zeros with multiplicity greater than $2$.
\end{theorem}

The solvable case of \cref{thm:degree3} can be proved quite easily by the results of Foote and Wales \cite{footeorder2}, though the non-solvable case does not appear to be easily addressed by Stark-like theorems. We elaborate on this in \cref{sec:degree-3-proof}.

In forthcoming work \cite{hkms}, we apply \cref{thm:mainthm} to establish that an analogue of the Mertens conjecture fails for certain number fields. This is our main~motivation to work with the order of zeros of Dedekind zeta functions.

\subsection*{Acknowledgements}
We are deeply grateful to Peter Humphries for supervising this project and to Ken Ono for his valuable suggestions. We would also like to thank Robert Lemke Oliver and Samit Dasgupta for helpfully directing us to the work of Stark. Finally, we are grateful for the generous support of the National Science Foundation (Grants DMS 2002265 and DMS 205118),
National Security Agency (Grant H98230-21-1-0059), the Thomas Jefferson Fund at the University of Virginia, and the Templeton~World Charity Foundation.

\section{Preliminaries}\label{preliminaries}

First of all, we recall the definition of an Artin $L$-function and some of its properties. These can be found in \cite[pp.~211,~220--222]{heilbronn1967zeta}.

\begin{definition}\label{def:artin-L} Given a Galois extension $L/K$ of number fields and a (complex linear)~representation $(\rho,V)$ of $\operatorname{Gal}(L/K)$ with character $\chi$, the \emph{Artin $L$-function} $L(s,\chi,L/K)$ is defined as a product of local factors, one for each prime ideal $\idp\subset\oO_K$. For an unramified prime $\idp$,~the factor is
\begin{equation*}
\det\left(I-N(\idp)^{-s}\rho(\frob(\idp))\right)^{-1},
\end{equation*}
where $\frob(\idp)$ is the Frobenius element of $\idp$ defined up to conjugacy in $\operatorname{Gal}(L/K)$. For~ramified $\idp$, the matrix is restricted to the subspace of $V$ fixed by the inertia group of $\idp$. As $\frob(\idp)$ is only defined up to an element of the inertia group, the restriction and corresponding determinant are only well-defined on this subspace. 
\end{definition}

\begin{lemma}\label{lem:artin-L}
Let $L/K$ be a Galois extension of number fields with Galois group $G$.
\begin{enumerate}[label=(\roman*), font=\normalfont]
\item[(a)] If $\chi$ is a one-dimensional character of $G$, then $L(s,\chi,L/K)$ is a Hecke $L$-function~and thus holomorphic on the whole complex plane $\C$, unless $\chi$ is the trivial character of $G$ in which case it is holomorphic except for a pole at $s = 1$.
\item[(b)] If $\chi$ is a character of some subgroup $H \subset G$, then $L(s, \Ind_{H}^{G} \chi, L/K) = L(s, \chi, L/L^{H})$.
\item[(c)] If $1$ is the trivial character of $G$, then $L(s, 1, L/K) = \zeta_{K}(s)$, and if $r_{G}$ is the character corresponding to the regular representation of $G$, then $L(s, r_{G}, L/K) = \zeta_{L}(s)$.
\item[(d)] If $\chi_{1}$ and $\chi_{2}$ are characters of $G$, then $L(s, \chi_{1}+\chi_{2}, L/K) = L(s, \chi_{1}, L/K) L(s, \chi_{2}, L/K)$.
\end{enumerate}
\end{lemma}

As a consequence, we have an explicit factorization of Dedekind zeta~functions.
\begin{corollary}
\label{cor:dedekind-zeta-factorization}
Let $L/K$ be a Galois extension of number fields. Then 
\begin{equation*}
\zeta_L(s) = \zeta_K(s) \prod_{\chi} L(s, \chi, L/K)^{\dim \chi},
\end{equation*}
where the product runs over all nontrivial irreducible characters $\chi$ of $\Gal(L/K)$.
\end{corollary}

Another corollary is the entireness of certain Artin $L$-functions.
\begin{corollary}\label{cor:hol-artin-induct}
Let $L/K$ be a Galois extension of number fields with Galois group $G$ and $\chi$ be a character of $G$ induced from a nontrivial one-dimensional character $\psi$ of a subgroup $H$ of $G$. Then $L(s, \chi, L/K)$ is entire.
\end{corollary}

\begin{proof}
This is a direct consequence of \cref{lem:artin-L}~(a) and (b).
\end{proof}

We may also combine \cref{cor:hol-artin-induct} with a representation-theoretic theorem to obtain some stronger results on the entireness of certain $L$-functions. For a group $G$, we let $r_G$ denote the character of the regular representation of $G$ and $1=1_G$ denote the character of the trivial representation of $G$. 

\begin{lemma}[Aramata--Brauer~\cite{aramata}]\label{lem:aramata-brauer}
Let $G$ be a finite group. There exist positive rational~numbers $\lambda_{i}$ and characters $\chi_{i}$ of $G$ such that
\begin{equation*}
r_{G} = 1+\sum_{i} \lambda_{i} \chi_{i},
\end{equation*}
where each $\chi_{i}$ is the induction of a one-dimensional character of some cyclic subgroup of $G$.
\end{lemma}

\begin{corollary}\label{cor:hol-quotient}
If $L/K$ is a Galois extension of number fields, $\zeta_{L}(s)/\zeta_{K}(s)$ is holomorphic.
\end{corollary}

\begin{proof}
Since Dedekind zeta functions are holomorphic (except for a pole of order $1$ at $s = 1$), the~quotient $\zeta_L(s)/\zeta_K(s)$ is meromorphic. To prove that it is holomorphic, we need only show that it has no poles.~Let $G = \mathrm{Gal}(L/K)$. Then, if
\begin{equation*}
r_{G}-1 = \sum_{i} \lambda_{i} \chi_{i},
\end{equation*}
\cref{lem:artin-L} renders that
\begin{equation*}
\left(\frac{\zeta_{L}(s)}{\zeta_{K}(s)}\right)^{N} = \prod_{i} L(s, \chi_{i}, L/K)^{N \lambda_{i}},
\end{equation*}
where $N$ is a positive integer such that $N \lambda_{i} \in \Z$ for each $i$. Via \cref{cor:hol-artin-induct}, the right-hand side is a holomorphic function, hence $\zeta_{L}(s)/\zeta_{K}(s)$ is as well.
\end{proof}

\begin{remark}\label{rmk:easy-sol-idea}
If some $\lambda_i$ in the decomposition of \cref{lem:aramata-brauer} exceeds $1$, then \cref{thm:mainthm} has a more direct proof, as each of the infinitely many nontrivial zeros of $L(s,\chi_i,L/K)$ would have multiplicity at least $\lambda_i$ and thus at least $\lceil\lambda_i\rceil\geq 2$. Unfortunately, at least in the explicit decomposition given in \cite{aramata}, this does not hold in general, even when identical characters are grouped appropriately.
\end{remark}

Before proving our main result, we need a class of groups where the Artin holomorphy~conjecture is known.

\begin{definition}
A group $G$ is \emph{monomial} if all of its irreducible characters are induced from characters of degree 1.
\end{definition}

By \cref{cor:hol-artin-induct}, the Artin holomorphy conjecture is known for all monomial Galois extensions of number fields. The following lemma due to Huppert presents an easily verifiable sufficient criterion for a group to be monomial which will be used subsequently.

\begin{lemma}[Huppert~{\cite{huppert-m-groups}}]
\label{lem:huppert}
Let $G$ be a finite group and let $N \lhd G$ be a proper normal subgroup for which $N$ is solvable and $G/N$ is supersolvable. If all Sylow subgroups of $N$ are abelian,~then~$G$ is monomial.
\end{lemma}

We also need a standard group-theoretic lemma. For completeness, we reproduce the~proof.

\begin{lemma}\label{lem:simple-group-nonabelian-subgp}
Any finite simple group $G$ besides $\Z/p\Z$ has a nonabelian proper subgroup.
\end{lemma}

\begin{proof}
Assume that this is not the case. Consider any maximal proper subgroup $H$ of $G$. Its normalizer is contained between $H$ and~$G$, so it must be either $H$ or $G$. If it were $G$, then $H$ would be normal, contradicting the simplicity~of $G$, so $H$ must be its own normalizer. If $H_{1}$ and $H_{2}$ are any two distinct maximal proper~subgroups of $G$, the normalizer of $H_{1} \cap H_{2}$ must contain both $H_{1}$ and $H_{2}$, since both $H_{1}$ and $H_{2}$ are abelian. As a consequence, it must be $G$ itself; this gives that $H_1\cap H_2$ is the trivial subgroup, as otherwise would contradict the simplicity of $G$. Hence, any maximal proper subgroup $H$ of $G$ has $|G|/|H|$ distinct conjugates, the union of which comprises exactly
\begin{equation*}
1+\frac{|G|}{|H|}(|H|-1)=|G|-\frac{|G|}{|H|}+1
\end{equation*}
elements. Since every group of non-prime order has a proper nontrivial subgroup, there is~some element $x\in G$ not counted in any conjugate of $H$. Thus there must exist some~maximal~proper subgroup $H'$ (the maximal proper subgroup containing $x$, say) of $G$ that is not a conjugate of $H$. The conjugates of this subgroup comprise $|G|-|G|/|H'|+1$ elements, of which only~the identity can be in any conjugate of $H$. Then we have that
\begin{equation*}
|G| \geq \left(|G|-\frac{|G|}{|H|}+1 \right)+\left(|G|-\frac{|G|}{|H'|}+1 \right)-1,
\end{equation*}
which implies $\min(|H|, |H'|) < 2$, a contradiction.
\end{proof}

\section{A Theorem of Stark}\label{sec:stark}

Stark \cite{stark1974some}, Foote and Murty \cite{foote_murty_1989}, and Foote and Wales \cite{footeorder2} considered the~holomorphy of Artin $L$-functions at a given point. Just as the Artin holomorphy conjecture implies the existence of higher order nontrivial zeros of Dedekind zeta functions, this more local phenomenon may be used to establish the existence of higher order zeros in certain circumstances. To elucidate this point, we recall a theorem of Stark to produce a first proof of \cref{thm:mainthm}.

\begin{lemma}[{Stark \cite[Theorem 3]{stark1974some}}]
\label{thm:stark}
Let $L/K$ be a Galois extension of number fields. Let $\rho \neq 1$ be such that the order of vanishing of $\zeta_L(s)$ at $s = \rho$ is at most 1. Then $L(s, \chi, L/K)$~is holomorphic at $s = \rho$ for all characters $\chi$.
\end{lemma}

The proof of this theorem is largely representation theoretic, using Frobenius reciprocity, \cref{lem:artin-L} and \cref{cor:hol-quotient} to show that the virtual character
$$\theta \coloneqq \sum_{\chi} \chi \cdot \ord_{s = \rho}L(s, \chi, L/K)$$
is a genuine character, and hence $L(s, \chi, L/K)$ is holomorphic at $s = \rho$ for all $\chi$. We now~give an initial proof of \cref{thm:mainthm}.

\begin{proof}[First Proof of \cref{thm:mainthm}]
Assume that $L/K$ is Galois but not abelian. Then $\Gal(L/K)$ has an irreducible representation $\chi$ of degree at least 2. Let $\rho$ be a zero or pole of $L(s, \chi, L/K)$ in the critical strip. It is known that infinitely many such $\rho$ exist.

Suppose for the sake of contradiction that $\ord_{s = \rho} \zeta_L(s) \leq 1$. Then by \cref{thm:stark}, the Artin $L$-function $L(s, \chi, L/K)$ is holomorphic at $\rho$ for each $\chi$. In particular, $L(\rho, \chi, L/K) = 0$. By \cref{cor:dedekind-zeta-factorization}, $\zeta_L(s)$ has a zero of order at least $\chi(1) > 1$ at $\rho$. Thus, $\zeta_L(s)$ has infinitely many nontrivial zeros of order at least 2.
\end{proof}

\section{A New Proof of \cref{thm:mainthm}}\label{sec:Proof-of-Theorem-1.1}

Let $\mathcal{S}_n$ be the set of finite groups $G$ with the property that for any Galois extension $L/K$ of number fields with Galois~group $G$, the Dedekind zeta function $\zeta_{L}(s)$ has infinitely many nontrivial zeros with multiplicity at least $n$. We establish by contradiction that all nonabelian groups $G$ are in $\mathcal S_2$. First of all, we show that this holds for all nonabelian monomial groups.

\begin{lemma}\label{lem:monomial-in-s}
Let $G$ be a finite nonabelian monomial group. Then $G \in \mathcal{S}_2$.
\end{lemma}

\begin{proof}
Suppose that $L/K$ is a Galois extension of number fields with monomial Galois group $G$. Then, by \cref{cor:dedekind-zeta-factorization},
\begin{equation*}
\zeta_{L}(s) = L(s, r_{G}, L/K) = \zeta_K(s) \cdot \prod_{\chi} L(s, \chi, L/K)^{\dim \chi},
\end{equation*}
where the product is over all nontrivial irreducible characters $\chi$ of $G$. Since $G$ is nonabelian, some such $\chi$ has degree greater than $1$. By \cref{cor:hol-artin-induct} and the definition of a monomial group,~each $L(s, \chi, L/K)$ is an entire~function with infinitely many nontrivial zeros. If $\dim \chi > 1$, then the infinitely many nontrivial zeros~of $L(s, \chi, L/K)$ occur with multiplicity at least $2$ as zeros of $\zeta_{L}(s)$. Hence we conclude that $G \in \mathcal{S}_2$.
\end{proof}

Next, we show that the property of a group belonging to $\mathcal{S}_n$ is induced through its subgroups and quotients by its normal subgroups.

\begin{lemma}\label{lem:subgp-and-quotient}
Let $G$ be a finite group and $n$ be any positive integer.
\begin{enumerate}[label=(\arabic*), font=\normalfont]
    \item If $H$ is a subgroup of $G$ and $H \in \mathcal S_n$, then $G \in \mathcal S_n$.
    \item If $N \lhd G$ is a normal subgroup and $G/N \in \mathcal S_n$, then $G \in  \mathcal S_n$.
\end{enumerate}
\end{lemma}
\begin{proof}
Let $L/K$ be any Galois extension of number fields with $\Gal(L/K) = G$.
\begin{enumerate}
    \item If $H \in \mathcal{S}_n$ is a subgroup of $G$, then $L/L^{H}$ is a Galois extension with Galois group~$H$. Since $H \in \mathcal{S}_n$, this means that $\zeta_{L}(s)$ has infinitely many nontrivial zeros with multiplicity at least $n$. Hence we have $G \in \mathcal{S}_n$.
    \item If $N \lhd G$ is a normal subgroup for which $G/N \in \mathcal{S}_n$, then $L^{N}/K$ is a Galois extension with Galois group $G/N$. The Dedekind zeta function $\zeta_{L^N}(s)$ thus has infinitely many nontrivial zeros of~multiplicity at least $n$. By \cref{cor:hol-quotient},
    \begin{equation*}
        \frac{\zeta_{L}(s)}{\zeta_{L^{N}}(s)}
    \end{equation*}
    is holomorphic, meaning that $\zeta_{L}$ also has infinitely many nontrivial zeros of multiplicity at least $n$. Hence we have $G \in \mathcal{S}_n$. \qedhere
\end{enumerate}
\end{proof}

We are now ready to establish~\cref{thm:mainthm}. 
\begin{proof}[Second Proof of Theorem \ref{thm:mainthm}]
Suppose for the sake of contradiction that there is a finite nonabelian group not in $\mathcal S_2$; let $G$ be such a group of minimal order.

By \cref{lem:subgp-and-quotient}, such a group~$G$ may only have abelian subgroups or quotients, as otherwise this would contradict the minimality of $G$. However, by \cref{lem:simple-group-nonabelian-subgp}, $G$ cannot be simple,~since nonabelian simple groups contain some nonabelian proper subgroup. Therefore, $G$ must have some nontrivial proper normal subgroup; take $N$ to be such a subgroup. Both of $N$ and $G/N$ are abelian, meaning that they are supersolvable. Hence, by \cref{lem:huppert}, $G$ is monomial,~which by \cref{lem:monomial-in-s} means $G \in \mathcal{S}_2$, as desired.
\end{proof}

\section{Proof of~\cref{thm:degree3}}\label{sec:degree-3-proof}

Although \cref{thm:mainthm} is strictly weaker than Stark's theorem, the method demonstrated in \cref{sec:Proof-of-Theorem-1.1} is much more amenable to collections of less well-behaved groups, like the collection of finite non-solvable groups. We will utilize this to prove \cref{thm:degree3} in the non-solvable~case.

First, however, we address the case where the Galois group is solvable; this proof may be completed in two ways. One method exploits the straightforward order-two generalization of Stark's result as proven by Foote and Wales. 
\begin{lemma}[{Foote--Wales \cite{footeorder2}}]\label{FooteWales}
Let $L/K$ be a solvable extension of number~fields and let $\rho \neq 1$ be such that the order of vanishing of $\zeta_L(s)$ at $s = \rho$ is at most $2$. Then $L(s, \chi, L/K)$ is holomorphic at $s = \rho$ for all characters $\chi$.
\end{lemma}
With Lemma~\ref{FooteWales} in mind, one replicates the proof in \cref{sec:stark} directly to show the~solvable case of \cref{thm:degree3}. Alternatively, it is also possible to treat the solvable case using methods similar to those in \cref{sec:Proof-of-Theorem-1.1}. We now give the proof.

\begin{proof}[Second proof of the solvable case of \cref{thm:degree3}]
For the sake of contradiction, we consider a solvable group $G\not\in \mathcal S_3$ with an irreducible representation of dimension at least $3$, and assume that $G$ is of minimal order subject to these constraints. Since any subgroup or quotient of a solvable group is solvable, \cref{lem:subgp-and-quotient} along with the minimality of $G$ implies that $G$ has no subgroup or quotient with an irreducible representation of dimension greater than $2$. 

As a result, $G$ must possess a normal subgroup $N$ with $G/N\cong \Z/p\Z$ for some prime $p$, and via minimality $N$ may only have irreducible representations of dimension $1$ or $2$. Such groups are explicitly classified in \cite[Theorem 3]{amitsur}; each has an abelian normal subgroup with abelian quotient, and so $N$ is monomial by \cref{lem:huppert}. This means that Artin holomorphy holds for all Artin characters of $L/L^N$.

Given a positive integer $n$, define the auxiliary meromorphic $L$-functions
$$L_n(s, L/K) \coloneqq \prod_{\dim \chi = n} L(s, \chi, L/K).$$
Using Clifford theory, the induction $\Ind_N^G(\psi)$ of any irreducible character $\psi$ of $N$ splits into irreducible characters of $\chi$ of equal degree (either $\dim\psi$ or $p\dim\psi$), so one can write $L_n(s, L/K)^n$ as a product of Artin $L$-functions whose characters are induced from $N$. Hence, $L_n(s, L/K)^n$ is holomorphic away from $s=1$ for each positive integer $n$; since $L_n(s, L/K)$ is meromorphic, this implies that $L_n(s, L/K)$ is holomorphic away from $s=1$. Moreover, $L_n(s, L/K)$ is itself an Artin $L$-function, and hence has infinitely many nontrivial zeros. Observing that $\zeta_L(s) = \prod_n L_n(s, L/K)^n$, it then follows that $\zeta_L(s)$ has infinitely many zeros with multiplicity at least 3.
\end{proof}

Hence, only the non-solvable case remains. In this context, we shall see that minimal~counterexamples would be minimal simple groups.

\begin{definition}\label{def:minimal-simple}
A \emph{minimal simple group} is a nonabelian finite simple group such that all proper subgroups are solvable.
\end{definition}

The classification of minimal simple groups was completed by Thompson. In what follows, $\PSL_n(q)$ denotes the projective special linear group of degree $n$ over the field $\F_q$ and $\Sz(2^{2k+1})$ denotes a Suzuki group.

\begin{lemma}\textup{(Thompson \cite[Corollary 1]{classification-of-minimal-simple-groups})}\label{lem:classification-of-minimal-simple-groups}
Let $G$ be a minimal simple group. Then $G$ is isomorphic to one of the following:

\begin{enumerate}[before=\normalfont]
    \item $\PSL_2(2^p)$ for some prime $p$.
    \item $\PSL_2(3^p)$ for some odd prime $p$.
    \item $\PSL_2(p)$ for some $p > 3$ prime where $p \equiv 2, 3 \mod 5$.
    \item $\PSL_3(3)$.
    \item $\Sz(2^p)$  for some odd prime $p$.
\end{enumerate}
\end{lemma}

Our immediate goal is to show that all such groups belong to $\mathcal{S}_3$. For this, we will require the following standard facts concerning the groups presented in \cref{lem:classification-of-minimal-simple-groups}. Here, $\AGL_1(q)$ signifies the affine group over $\F_q$ as defined in \eqref{eq:aff-gp-def}, and $\AGL'_1(q)$ denotes the subgroup of $\AGL_1(q)$ formed by restricting the entry $a$ in \eqref{eq:aff-gp-def} to those elements of $\F_q^\times$ that are squares.

\begin{lemma}\label{lem:group-theoretic-lemmas}
The following statements hold:
\begin{enumerate}[before=\normalfont]
    \item $\PSL_2(3) \cong \AGL_1(4)$.
    \item $\Sz(2) \cong \AGL_1(5)$.
    \item $\AGL_1(2^n) \leq \PSL_2(2^n)$ for any $n$.
    \item $\AGL_1'(q) \leq \PSL_2(q)$ for $q = p^n$ odd. 
\end{enumerate}

\begin{proof}
Parts (1) and (2) are routine calculations. For parts (3) and (4), consider the subgroup of upper-triangular matrices ${\SUT_2(q) \leq \SL_2(q)}$ and its image $\PSUT_2(q) \leq \PSL_2(q)$. Observe that
$$
\PSUT_2(q) = \left\{\pm a^{-1}\begin{pmatrix}
a^2 & b \\
0 & 1
\end{pmatrix} \bigg\vert a \in \F_q^\times, b \in \F_q \right\}
$$
Noting the similarities in the definitions of $\PSUT_2(q)$ and $\AGL_1(q)$, we see that $\PSUT_2(q) \cong \AGL'_1(q)$. When $q = 2^n$, the map $a\mapsto a^2$ is the Frobenius automorphism on $\F_q^\times$, and we may further confirm that ${\PSUT_2(q) \cong \AGL_1(q)}$.
\end{proof}
\end{lemma}

Browkin \cite{browkin-2013} establishes that $\AGL_1(q) \in \mathcal{S}_{q - 1}$. We extend this result to the index two subgroup $\AGL'_1(q)$.

\begin{lemma}\label{lem:semidirect-computation}
Let $q = p^n$ for $p \geq 3$. Then $\AGL'_1(q) \in \mathcal{S}_{(q - 1)/2}$.

\begin{proof}
Observe that $[\AGL_1(q): \AGL'_1(q)] = 2$. In particular, this means $\AGL'_1(q)$ is normal in $\AGL_1(q)$. By computations of Browkin \cite{browkin-2013}, $\AGL_1(q)$ is monomial with an irreducible representation $\chi$ of degree $q - 1$. $\AGL'_1(q)$ is also monomial by \cref{lem:huppert}, and Clifford~theory gives that $\chi$ restricts to a representation of $\AGL'_1(q)$ whose constituent irreducible representations have degree at least $(q - 1)/2$. Hence, $\AGL'_1(q) \in \mathcal{S}_{(q-1)/2}$.
\end{proof}
\end{lemma}

Note that $\mathcal{S}_n \subseteq \mathcal{S}_3$ for $n \geq 3$. Using these results, we show that all minimal simple groups belong to $\mathcal{S}_3$:

\begin{lemma}\label{lem:minimal-simple-case}
Let $G$ be a minimal simple group. Then $G \in \mathcal{S}_3$.

\begin{proof}
By \cref{lem:classification-of-minimal-simple-groups}, there are five cases to prove. In each case, by \cref{lem:subgp-and-quotient}, it suffices to find a subgroup belonging to $\mathcal{S}_3$. Such subgroups are given in \cref{lem:group-theoretic-lemmas}. In particular,
\begin{enumerate}
    \item $\AGL_1(2^p) \leq \PSL_2(2^p)$ where $p$ prime.
    \item $\AGL'_1(3^p) \leq \PSL_2(3^p)$ where $p$ is odd and prime.
    \item $\AGL'_1(p) \leq \PSL_2(p)$ where $p \geq 7$.
    \item $\AGL_1(4) \cong \PSL_2(3) \leq \PSL_3(3)$.
    \item $\AGL_1(5) \cong \Sz(2) \leq \Sz(2^p)$.
\end{enumerate}
From the work of Browkin \cite{browkin-2013}, one has $\AGL_1(q) \in \mathcal{S}_3$ when $q \geq 4$. By \cref{lem:semidirect-computation}, it follows that $\AGL_1'(q) \in \mathcal{S}_3$ for $q \geq 7$. In all cases, we may conclude that $G \in \mathcal{S}_3$.
\end{proof}
\end{lemma}

We are now ready to complete the proof of \cref{thm:degree3}. Since we have already shown the result in the case where $\Gal(L/K)$ is solvable, we need only treat the non-solvable case.

\begin{proof}[Proof of \cref{thm:degree3}]
Suppose for the sake of contradiction that there is a finite non-solvable group not in $\mathcal S_3$; let $G$ be such a group of minimal order.

By \cref{lem:subgp-and-quotient} and the minimality of $G$, it follows that all proper subgroups and nontrivial quotients of $G$ are solvable. We initially wish to show that $G$ is simple. Suppose for the sake of contradiction that $G$ is not simple. Then $G$ has a maximal proper nontrivial normal subgroup $N$, so that $G/N$ is simple. If $G/N$ is abelian, by the solvability of $N$, $G$ is also solvable, contradicting our initial assumptions on $G$. If $G/N$ is nonabelian, then $G/N$ is non-solvable, contradicting the fact that $G$ only has solvable nontrivial quotients. Thus, no such $N$ exists, meaning $G$ is simple.

Since all proper subgroups of $G$ are solvable, $G$ is a minimal simple group, which by \cref{lem:minimal-simple-case} means $G \in \mathcal{S}_3$ as desired.
\end{proof}

\section{Conclusion and conjectures}\label{conclusion-and-conjectures}
Theorems \ref{thm:mainthm} and \ref{thm:degree3} show that, up through order $3$, the predictions of the Artin holomorphy conjecture on the orders of zeros of $\zeta_L(s)$ for Galois $L/K$ hold unconditionally. However, Artin holomorphy implies results that are both broader (which apply to non-Galois extensions) and stronger (which guarantee even larger multiplicities). To this end, we present two conjectures as possible extensions of our work.

\begin{conjecture} Let $L/K$ be an extension of number fields, $M$ be the Galois closure of $L/K$ and write $H=\Gal(L/K)\subset G=\Gal(M/K)$. If $\Ind_H^G(1_H)$ contains in its decomposition an irreducible representation of $G$ with nontrivial multiplicity, then $\zeta_L(s)$ has infinitely many nontrivial zeros of multiplicity greater than $1$. 
\end{conjecture}

\cref{thm:mainthm} implies this conjecture in some non-obvious cases, e.g.~whenever there~is some $K\subset K'\subset L$ with $L/K'$ a nonabelian Galois extension. Nevertheless, there are cases in~which $\Ind_H^G(1_H)$ may have irreducible components with nontrivial multiplicity even if $H$ is a maximal subgroup of $G$, i.e.~if there are no fields between $K$ and $L$. These cases should be~specifically difficult to treat using methods similar to ours, since there is no obvious way to replace the extension by a smaller one. 

\begin{conjecture} \label{conj:higher-mult} Let $L/K$ is a Galois extension of number fields and $G=\operatorname{Gal}(L/K)$. If $G$ has an irreducible representation of degree $m$, then $\zeta_L(s)$ has infinitely many nontrivial zeros of multiplicity at least $m$.
\end{conjecture}

This is known when $G$ is a monomial group, in which case the necessary special cases of the Artin holomorphy conjecture hold unconditionally. On the other hand, both Theorems \ref{thm:mainthm} and \ref{thm:degree3} use particular information about groups with representations of small dimension. To generalize this method to larger $m$, more work is necessary in studying groups with representations of bounded dimension.

\bibliographystyle{alpha}
\bibliography{HigherOrderZeros}

\begin{thebibliography}{HKMS}

\bibitem[Ami61]{amitsur}
S.~A. Amitsur.
\newblock Groups with representations of bounded degree {II}.
\newblock {\em Illinois Journal of Mathematics}, 5(2):198 -- 205, 1961.

\bibitem[Ara33]{aramata}
Hideo Aramata.
\newblock {\"U}ber die {T}eilbarkeit der {D}edekindschen {Z}etafunktionen.
\newblock {\em Proceedings of the Imperial Academy. Tokyo}, 9(2):31--34, 1933.

\bibitem[Bro13]{browkin-2013}
Jerzy Browkin.
\newblock Multiple zeros of {D}edekind zeta functions.
\newblock {\em Functiones et Approximatio Commentarii Mathematici},
  49(2):383--390, 2013.

\bibitem[FM89]{foote_murty_1989}
Richard Foote and V.~Kumar Murty.
\newblock Zeros and poles of {A}rtin {$L$}-series.
\newblock {\em Mathematical Proceedings of the Cambridge Philosophical
  Society}, 105(1):5–11, 1989.

\bibitem[FW90]{footeorder2}
Richard Foote and David Wales.
\newblock Zeros of order 2 of {Dedekind} zeta functions and {Artin's}
  conjecture.
\newblock {\em Journal of Algebra}, 131(1):226--257, 1990.

\bibitem[Hei67]{heilbronn1967zeta}
Hans Heilbronn.
\newblock Zeta-functions and {$L$}-functions.
\newblock In John William~Scott Cassels and Albrecht Fr{\"{o}}hlich, editors,
  {\em Algebraic number theory (Proc. Instructional Conf., Brighton, 1965)}.
  Thompson, Washington, D.C., 1967.

\bibitem[HKMS]{hkms}
Daniel Hu, Ikuya Kaneko, Spencer Martin, and Carl Schildkraut.
\newblock On the {M}ertens conjecture over number fields.
\newblock In preparation.

\bibitem[Hup53]{huppert-m-groups}
Bertram Huppert.
\newblock Monomiale {D}arstellung endlicher {G}ruppen.
\newblock {\em Nagoya Mathematical Journal}, 6:93--94, 1953.

\bibitem[Sta74]{stark1974some}
Harold~M Stark.
\newblock Some effective cases of the {B}rauer-{S}iegel theorem.
\newblock {\em Inventiones mathematicae}, 23(2):135--152, 1974.

\bibitem[Tho68]{classification-of-minimal-simple-groups}
John~G. Thompson.
\newblock {Nonsolvable finite groups all of whose local subgroups are
  solvable}.
\newblock {\em Bulletin of the American Mathematical Society}, 74(3):383 --
  437, 1968.

\end{thebibliography}

\end{document}